\documentclass[12pt]{article}
\usepackage{a4wide}

\usepackage[a4paper, margin=2.3cm]{geometry}
\usepackage{amsmath,amssymb,amsthm}
\usepackage{color}
\usepackage{parskip}
\usepackage{hyperref}
\usepackage{parskip}
\usepackage{setspace}
\usepackage{graphicx}

\newtheorem{theorem}{Theorem}[section]
\newtheorem{remark}{Remark}[section]

\newtheorem{lemma}[theorem]{Lemma}
\newtheorem{proposition}[theorem]{Proposition}
\newtheorem{claim}[theorem]{Claim}

\newtheorem{definition}[theorem]{Definition}

\newtheorem*{definition*}{Definition}

\def\R{\mathbb{R}}
\newcommand{\Int}{\textrm{Int}}

\begin{document}
\title{Pinned simplices and connections to product of sets on paraboloids}
\author{Alex Iosevich\thanks{Department of Mathematics, University of Rochester, USA. Email: alex.iosevich@rochester.edu}
\and Minh-Quy Pham\thanks{Department of Mathematics, University of Rochester, USA. Email: qpham3@ur.rochester.edu}
\and Thang Pham\thanks{University of Science, Vietnam National University, Hanoi. Email: thangpham.math@vnu.edu.vn}
\and Chun-Yen Shen\thanks{Department of Mathematics, National Taiwan University, Taiwan. Email: cyshen@math.ntu.edu.tw}}
\maketitle

\begin{abstract} In this paper we obtain improved dimensional thresholds for dot product sets corresponding to compact subsets of a paraboloid. As a direct application of these estimates, we obtain significant improvements to the best known dimensional thresholds that guarantee that a given compact subset of Euclidean space determines a positive proportion of all possible congruence classes of simplexes. In many regimes this improves the results previously obtained by Erdogan-Hart-Iosevich (\cite{EHI}), Greenleaf-Iosevich-Liu-Palsson (\cite{GILP}) and others. 

\end{abstract}

\section{Introduction}
In a very recent paper, a connection between isosceles dot-product triangles in a subset of a paraboloid and isosceles distance triangles in a subset of lower dimensional vector space over finite fields has been discovered by Chang, Mohammadi, Pham, and Shen in \cite{Chang}. More precisely, assume that we have three points $a=(a_1, \ldots, a_{d}, a_{d+1}), b=(b_1, \ldots, b_d, b_{d+1}), c=(c_1, \ldots, c_d, c_{d+1})$ on the paraboloid defined by $x_{d+1}=x_1^2+\cdots+x_d^2$, the equation 
\[a\cdot b=a\cdot c,\]
can be written as
\begin{equation}\label{eq:triangle}
\left\vert\frac{-\overline{a}}{2|\overline{a}|^2}-\overline{b}\right\vert=\left\vert\frac{-\overline{a}}{2|\overline{a}|^2}-\overline{c}\right\vert,\end{equation}
where $\overline{x}=(x_1, \ldots, x_d)$ for any $x$ on the paraboloid, and $|\overline{x}|^2=x_1^2+\cdots+x_d^2$.

This tells us that if $a\cdot b=a\cdot c$, then we have an isosceles triangle with the apex $\frac{-\overline{a}}{2|\overline{a}|}$ and the base with two vertices $\overline{b}$ and $\overline{c}$ in the vector space of one dimension lower. As a consequence of this identity, they are able to provide a link between the Erd\H{o}s-Falconer distance problem and the problem of products of sets in paraboloids. As a consequence, several improvements of well-known results in the literature have been obtained. 

The main purpose of this paper is to study applications of the connection described above in the continuous (Euclidean) setting. More precisely, we will focus on the distribution of dot-product values determined by sets on the truncated paraboloid and the structures of pinned simplices in arbitrary sets, where the truncated paraboloid in $\R^d$ is defined by 
$$P_d:=\{(\overline{x}, \vert \overline{x}\vert^2)\in \R^d: |\overline{x}|\le 1\},$$ where ~$|\overline{x}|^2=x_1^2+\cdots+x_{d-1}^2$.

\subsection{Products of sets on paraboloids}
Before stating our main results, let us recall the following result due to Erdo\u{g}an, Hart, and Iosevich in \cite{EHI}.

\begin{theorem}[\cite{EHI}]
Let $\mathbb{S}^{d-1}$ be the unit sphere centered at the origin in $\mathbb{R}^d$. Let $E\subset \mathbb{S}^{d-1}$ satisfy 
$\dim_H(E)>\frac{d}{2}$. Assume that $\mu_E$ is a Frostman measure on $E$. Then 
\[\mathcal{L}^1(\{x\cdot y\colon x, y\in E\})=\mathcal{L}^1(\{|x-y|=\sqrt{(x_1-y_1)^2+\cdots+(x_d-y_d)^2}\colon x\in E\})>0\]
for almost $\mu$-every $y\in E$.
\end{theorem}

It is natural to expect the same dimensional threshold $\frac{d}{2}$ to hold for other varieties, for instance, paraboloids or cones. In the following results, using the above connection and recent developments on the Falconer distance problem, we are able to obtain improvements for paraboloids in $\mathbb{R}^d$. 

For $E\subset \mathbb{R}^d$ and $x\in E$, we define the product set and its pinned analog formed by $E$ at $x$, respectively, by
\begin{align*}
    \Pi(E):=E\cdot E=\{ x\cdot y: x,y\in E\}, ~\Pi_x^1(E):=\{ x\cdot y: y\in E\}.
\end{align*}

\begin{theorem}\label{thm1.1}
Let $E\subset \R^d$ be a compact set on $P_d$, $d\geq 3$. Assume that
\begin{align*}
    \dim_H(E)>\frac{d-1}{2}+\frac{1}{4}+\frac{1}{8d-12},
\end{align*}
then there exist $x\in E$ and a measure $\nu_x$ supported on $\Pi_x^1(E)$ such that $||\nu_x||_{L^2}<\infty$. In particular,  $\Pi_x^1(E)$ has positive Lebesgue measure.
\end{theorem}
In three dimensions, we are able to obtain a better dimensional threshold. 
\begin{theorem}\label{thm1.2}
Let $E\subset \R^3$ be a compact set on $P_3$. Assume that
\begin{align*}
    \dim_H(E)>\frac{5}{4},
\end{align*}
then there exists $x\in E$ such that $\Pi_x^1(E)$ has positive Lebesgue measure.
\end{theorem}

It is worth noting that, in general, one can not go lower than $\frac{d-1}{2}$. A construction will be provided in Section \ref{sharpness}. If we want to say something about the Hausdorff dimension of the pinned product sets or the property of having a non-empty interior, the next theorem provides some dimensional thresholds in this direction. 
\begin{theorem}\label{thm1.3}
Let $E\subset \R^d$ be a compact set on $P_d$, $d\geq 3$. 
\begin{itemize}
    \item[(i)] If $\dim_H(E)>\frac{d^2-2d+d\beta}{2(d-1)}$ for any $\beta>0$, then there exists $x\in E$ such that 
    \begin{align*}
        \dim_H(\Pi_x^1(E))\geq \beta.
    \end{align*}
    \item[(ii)] If $\dim_H(E)>\frac{d^2}{2(d-1)}$, then there exists $x\in E$ such that
    \begin{align*}
        \Int(\Pi_x^1(E))\neq \emptyset,
    \end{align*}
\end{itemize}
where $\Int(A)$ denotes the interior of $A$.
\end{theorem}

\subsection{Pinned simplex problem}
\begin{definition}
Let $d\geq 2$ and $1\leq k\leq d$. Let $E\subset \R^d$, we define the set of distinct congruent $k$-simplices with vertices in $E$ to be
\begin{align*}
    T_k(E):=E^{k+1}/\sim,
\end{align*}
where $(x^1,\dots, x^{k+1})\sim (y^1,\dots, y^{k+1})$ if and only if there exist an element $g$ of the orthogonal group $O(d)$ and a translation $z\in \R^d$ such that
\begin{align*}
    y^j=gx^j+z,\quad \forall \, 1\leq j\leq k+1.
\end{align*}
\end{definition}
Note that we may view $T_k(E)$ as a subset of $\R^{\binom{k+1}{2}}$ via the map 
\begin{align*}
    [(x^1,\dots, x^{k+1})]\mapsto \big(\vert x^i-x^j\vert \big)_{1\leq i<j\leq k+1}.
\end{align*}
The question of finding the smallest dimensional threshold $\alpha$ such that the set $T_k(E)$ has positive Lebesgue measure whenever $\dim_H(E)>\alpha$ has received much attention during the two recent decades. When $k=1$, this is the Falconer distance problem, and the best current dimensional thresholds are due to Guth, Iosevich, Ou, and Wang \cite{GIOW} in two dimensions with $\frac{5}{4}$, Du, Iosevich, Ou, Wang and Zhang \cite{DIOWZ} in higher even dimensions with $\frac{d}{2}+\frac{1}{4}$, and Du and Zhang \cite{DZ} in odd dimensions with $\frac{d}{2}+\frac{d}{4d-2}$.

When $k>1$, we have three following results, where the first is due to Erdo\u{g}an, Hart, and Iosevich \cite{EHI}, the second is due to Grafakos, Greenleaf, Iosevich and Palsson in \cite{Grafakosetal}, and the third is due to Greenleaf, Iosevich, Liu, and Palsson in \cite{GILP}.

\begin{theorem}[pinned simplex, \cite{EHI}]
Let $E\subset \mathbb{R}^d$ be a compact set. Assume that 
\[\dim_H(E)>\frac{d+k+1}{2},\]
then there exists $x\in E$ such that the set of $k$-simplices pinned at $x$ has positive Lebesgue measure. 
\end{theorem}

\begin{theorem}[\cite{Grafakosetal}]
Let $E\subset \mathbb{R}^d$ be a compact set. Assume that 
\[\dim_H(E)>d-\frac{d-1}{2k},\]
then $\mathcal{L}^{\binom{k+1}{2}}(T_k(E))>0$.
\end{theorem}

\begin{theorem}[\cite{GILP}]
Let $E\subset \mathbb{R}^d$ be a compact set. Assume that 
\[\dim_H(E)>\frac{dk+1}{k+1},\]
then $\mathcal{L}^{\binom{k+1}{2}}(T_k(E))>0$.
\end{theorem}

In this paper, we are able to establish the following improvement in which the ideas described above play a crucial role. 
\begin{theorem}[pinned simplex]\label{thm_distance_simplex}
Let $E\subset \mathbb{R}^d$ be a compact set. Assume that 
\[\dim_H(E)>\frac{d+k}{2},\]
then there exists $x\in E$ such that the set of $k$-simplices pinned at $x$ has positive Lebesgue measure. 
\end{theorem}
\begin{remark}
For comparison, the exponent $\frac{d+k}{2}$ in Theorem \ref{thm_distance_simplex} is clearly smaller than the exponent $\frac{d+k+1}{2}$ obtained in \cite{EHI}, is better than exponent $d-\frac{d-1}{2k}$ in \cite{Grafakosetal} for $k< d-1$, and is lower than exponent $\frac{dk+1}{k+1}$ in \cite{GILP} when $k< d-2$. For the remaining cases that $k\geq d-2$, the exponent in \cite{GILP} is still the best known. 
\end{remark}

\begin{remark} It is not difficult to see that for the $k$-simplex problem, the dimensional exponent $k-1$ cannot be exceeded since the underlying set $E$ may be contained in a copy of ${\Bbb R}^{k-1}$ where a non-degenerate $k$-simplex cannot be constructed. 
\end{remark} 

\textbf{Notation.} Throughout the paper, we denote $X\lesssim Y$ if $X\leq CY$ for some constant $C>0$, and $X\approx Y$ iff $X\lesssim Y$ and $Y\lesssim X$. Given a set $A\subset \R^n$, its $n$-dimensional Lebesgue measure is denoted by $\mathcal{L}^{n}(A)$. We also denote $\R^\ast=\{x\in \R: x\neq 0\}$ as the set of nonzero real numbers. For a measure $\lambda$ and a map $f$, we denote $f_\sharp \mu$ as the pushforward of $\mu$ under the map $f$. For a set $A$, $C_0(A)$ will be used to denote the set of continuous functions which supported on $A$, and $C_0^\infty(A)$ is the subset of $C_0(A)$ which contains all smooth functions.

The paper is organized as follows: In Section \ref{sec_preliminaries}, we recall some definitions and results which are needed for later proofs. We establish a connection between product of sets on the paraboloid $P_d$ and the pinned simplex problem in Section \ref{sec_connection}. We prove Theorems \ref{thm1.1}, \ref{thm1.2}, and \ref{thm1.3} in Section \ref{sec_proof123}. Proof of Theorem \ref{thm_distance_simplex} (pinned simplex) will be presented in Section \ref{sec_proof5}. The last section is devoted for constructions and discussions. 

\vskip.125in 

\subsection{Acknowledgements} The first listed author was supported in part by the National Science Foundation grants no. HDR TRIPODS - 1934962 and NSF DMS 2154232. The fourth listed author Chun-Yen was supported in part by MOST, through grant 108-2628-M-002-010-MY4.

\vskip.125in 

\section{Preliminaries}\label{sec_preliminaries}
In this section, we recall some notation and results about Frostman measures, $L^2$ spherical averages, and pinned distance set which are need for our proofs in later sections. 
\subsection{Frostman measures and spherical averages}
\begin{lemma}[Frostman's lemma, Theorem 2.7, \cite{Mattilabook15}]
Let $0\leq s\leq n$. For a Borel set $E\subset \R^n$, the $s-$dimensional Hausdorff measure of $E$ is positive if and only if there exists a probability measure $\mu$ on $E$ such that
\begin{align*}
    \mu(B(x,r))\lesssim r^s, \quad \forall\, x\in \R^n,\,\,r>0.
\end{align*}
\end{lemma}
In particular, Frostman's lemma implies that given any exponent $s_\mu< \dim_\mathcal{H}(E)$, there exists a probability measure $\mu$ on $E$ such that 
\begin{align}\label{eq_Frostman}
    \mu(B(x,r))\lesssim r^{s_\mu}, \quad \forall\, x\in \R^n, r>0.
\end{align}
A measure satisfying \eqref{eq_Frostman} is often called an $s_{\mu}$-Frostman measure.

Given a probability measure $\mu$ with compact support on $\R^n$, $n\geq 2$, we define the $L^2$ spherical averages of Fourier transform of $\mu$ by
\begin{align*}
    \int_{S^{n-1}}\vert \widehat{\mu}(r\omega)\vert^2d\omega,
\end{align*}
where $\omega$ is the surface measure on the unit sphere $S^{n-1}$. In \cite{Mattila87}, Mattila developed a machinery to study Falconer's distance problem, by reducing the orginial problem to the study of the decay rates of spherical averages of fractal measures, namely, the supremum of the numbers $\beta$ for which:
\begin{align}\label{eq_spherical_average}
    \int_{S^{n-1}}\vert \widehat{\mu}(r\omega)\vert^2d\omega \lesssim r^{-\beta},\quad r>1.
\end{align}
It is proved that for any $s$-Frostman measure $\mu$, \eqref{eq_spherical_average} holds with
$$
    \beta(s)
    :=\left\{
    \begin{array}{lll}
    s, &s\in \big(0,\frac{n-1}{2}\big], &(\text{Mattila \cite{Mattila87}})\\
    \frac{n-1}{2}, &s\in \big[\frac{n-1}{2},\frac{n}{2}\big], &(\text{Mattila \cite{Mattila87}})\\
    \frac{n+2s-2}{4}, &s\in \big[ \frac{n}{2},\frac{n+1}{2}\big], &(\text{Wolff \cite{Wolff} \text{ and } Erdo\u{g}an \cite{Erdogan05}})
    \end{array}
    \right..
$$
Recently, the estimate \eqref{eq_spherical_average} was improved by Du–Guth–
Ou–Wang–Wilson–Zhang \cite{DGOWWZ} when $n\geq 3$ and Du-Zhang \cite{DZ} when $n\geq 4$, namely,
\begin{align*}
    \beta(s)= \frac{(n-1)s}{n},\quad s\in \bigg[\frac{n}{2},\frac{n+1}{2}\bigg].
\end{align*}
When $s$ is large, see Luc\`{a} and Rogers \cite{LR19}.

\begin{remark} There are certain arithmetic condition that limit the exponents one can obtain for spherical averages. See, for example, \cite{IosevichRudnev07} and \cite{BBCRV07} and the references contained therein. \end{remark} 

\subsection{Pinned distance set problem}
Given $E\subset \R^n$, $n\geq 2$, and a point $x\in \R^n$, we define the pinned distance set of $E$ at $x$ by
\begin{align*}
    \Delta_x(E)=\{ \vert x-y\vert: y\in E\}.
\end{align*}
The pinned distance set problem is a more restrictive version of the Falconer problem, which asks for the smallest dimensional threshold of a given compact set $E$ such that there exists $x\in E$ with $\mathcal{L}^1(\Delta_x(E))>0$. This first investigation was made by Peres and Schlag in \cite{Peres}. In a recent work \cite{LiuL2}, Liu proved the following result for pinned distance set problem, using an $L^2$ identity and the estimates of the Fourier decay \eqref{eq_spherical_average} above.

\begin{proposition}[Theorem 1.4, \cite{LiuL2}]\label{prop_Liu_pinned_distance}
Suppose $E\subset \R^n$, $n\geq 2$. Then
\begin{align*}
    \dim_H(\{x\in \R^n: \mathcal{L}^1(\Delta_x(E)) =0\})\leq 
    \inf\{s: \dim_H(E)+\beta(s)>n\}.
\end{align*}
In particular, if $\dim_H(E)+\beta(\dim_H(E))>n$, there exists $x\in E$ such that $\Delta_x(E)$ has positive Lebesgue measure.
\end{proposition}

In two dimensions, we recall the following result due to Guth, Iosevich, Ou and Wang \cite{GIOW}, which is the current best known dimensional threshold of the Falconer problem. The following variant is taken from \cite{OuTaylor}.

\begin{proposition}[\cite{GIOW}]\label{prop_outaylor}
Let $E,F\subset \R^2$ be a pair of compact sets with positive $s$-dimensional Hausdorff measure for some $s>\frac{5}{4}$. Further, suppose that there exist Frostman probability measures $\mu_E$ and $\mu_F$ on $E$ and $F$ respectively, with exponent $s$. Then 
\begin{align*}
    \mu_{F}(y\in F: \mathcal{L}^1(\Delta_y(E))>0)>0.
\end{align*}
\end{proposition}

We also recall the following result, which is proved in \cite{IosevichLiu19} by Iosevich and Liu.
\begin{proposition}[Theorem 1.3, \cite{IosevichLiu19}]\label{prop_IosevichLiu19}
Let $E, F\subset \R^n$ be compact sets. Then there exists a probability measure $\mu_F$ on $F$ such that  for $\mu_F-$a.e. $x\in F$,
\begin{itemize}
    \item[(i)] $\dim_H(\Delta_x(E))\geq \beta$ if $\dim_H(E)+\frac{n-1}{n+1}\dim_H(F)>n-1+\beta,$
    \item[(ii)] $\mathcal{L}^1(\Delta_x(E))>0$ if $\dim_H(E)+\frac{n-1}{n+1}\dim_H(F)>n$,
    \item[(iii)] $\Int(\Delta_x(E))\neq \emptyset$ if $\dim_H(E)+\frac{n-1}{n+1}\dim_H(F)>n+1$.
\end{itemize}
\end{proposition}

\section{Connection between product of sets on paraboloids and simplex problem}\label{sec_connection}

The main goal of this section is to set up the problem and outline a connection between the pinned dot-product and pinned simplex problems.

Assume that $E$ and $F$ are compact sets on $P_n$, with $n\geq 3$. By pigeon-holing, we may assume that
\begin{align*}
   \vert x\vert,\,\vert y\vert \geq c>0, \quad \forall\, x\in E, y\in F,
\end{align*}
for some constant $c>0$. Denote $P_n^{\ast}=P_n\setminus \{0\}$. We consider the following maps
\begin{align*}
    H_1: P_n^\ast &\to (\R^\ast)^{n-1},\quad  x\mapsto H_1(x)=\overline{x}, \\
    H_2: P_n^\ast &\to (\R^\ast)^{n-1}, \quad y\mapsto H_2(y)=\frac{-\overline{y}}{2\vert \overline{y}\vert^2},
\end{align*}
where for each point $x=(x_1, \dots, x_{n-1},x_n)\in \R^n$, we denote $\overline{x}=(x_1, \dots,x_{n-1})\in \R^{n-1}$. The images of $E$ and $F$ via $H_1$ and $H_2$, will be denoted by $E'$ and $F'$, respectively, i.e.
\begin{align}\label{eq_defE'F'}
    E':=H_1(E) \quad\text{ and }\quad
    F':=H_2(F).
\end{align}
Let's fix some point $y\in F$ and put $y'=H_2(y)$. Then using a similar argument as in the introduction, one can check that for all $x\in E$,
\begin{align}\label{eq_xyt}
    x\cdot y=t \quad \Longleftrightarrow\quad \vert y'-\overline{x}\vert=\vert y'\vert\sqrt{1+4t},
\end{align}
where $\overline{x}=H_1(x)$. Geometrically, the projection of all points $x$ satisfying the left hand side to one lower dimensional space is on a sphere centered at $y'$ of radius $\vert y'\vert\sqrt{1+4t}$. As we mentioned in the introduction, this geometric observation plays a vital role in our work.

%Explain: If $x\cdot y=t$ for some $t\geq 0$,  then we must have
%\begin{align*}
%   (\overline{y}, \vert \overline{y}\vert^2)\cdot (\overline{x}, \vert \overline{x}\vert^2)
%   =t
%  \quad \Leftrightarrow\quad  \bigg(\frac{\overline{y}}{\vert \overline{y}\vert^2},1\bigg)\cdot (\overline{x}, \vert \overline{x}\vert^2)=\frac{t}{\vert \overline{y}\vert^2}
%   \quad \Leftrightarrow\quad  \frac{\overline{y}\cdot \overline{x}}{\vert \overline{y}\vert^2}+\vert \overline{x}\vert^2=\frac{t}{\vert \overline{y}\vert^2}.
%\end{align*}
%Put $y'=H_2(y)\in F'$, noting that $\vert y'\vert =\vert H_2(y)\vert=\frac{\vert \overline{y}\vert}{2\vert\overline{y}\vert^2}=\frac{1}{2\vert \overline{y}\vert}$, then the last equality above is equivalent with
%\begin{align*}
%    \big\vert y'-\overline{x}\big\vert^2=\bigg\vert \frac{-\overline{y}}{2\vert \overline{y}\vert^2}-\overline{x}\bigg\vert^2
%    =\frac{1}{4\vert \overline{y}\vert^2}+\frac{\overline{y}\cdot \overline{x}}{\vert \overline{y}\vert^2}+\vert \overline{x}\vert^2
%    =\frac{1+4t}{4\vert \overline{y}\vert^2}=(1+4t)\vert y'\vert^2.
%\end{align*}
Motivated by this, let us consider more general case. Assume that $k$ is an integer such that $1\leq k\leq n$. Given any $k-$tuple $y=(y^1,\dots, y^k)\in F^k$, we define 
%the pinned $k-$dot-product stars formed by $E$ by
\begin{align*}
    \Pi_y^k(E):=\big\{ \,(x\cdot y^1,\dots, x\cdot y^k)\,\colon\, x\in E\,\big\}.
\end{align*}
Put $y'=(y^{1'},\ldots, y^{k'})\in (F')^k$, where $y^{j'}=H_2(y^j)$ for each $1\leq j\leq n$, and define the set of $k-$stars determined by $E'$ pinned at $y'$ by
\begin{align*}
    \Delta_{y'}^k(E'):=\big\{\, (\,\vert y^{1'}-\overline{x}\vert\,,\ldots, \vert y^{k'}-\overline{x}\vert\,)\,\colon \,\overline{x}\in E'\,\big\}.
\end{align*}
For each tuple $y\in F^k$, we define the corresponding map $\Gamma_y$ by
\begin{align*}
    \Gamma_y: \big[-1/4,\infty\big)^k &\to \big[0,\,\infty\,\big)^k\\
    (t_1,\dots,t_k)&\mapsto \big(\,\vert y^{1'}\vert\sqrt{1+4t_1}, \ldots,\vert y^{k'}\vert\sqrt{1+4t_k}\,\big).
\end{align*}
Then it is straightforward to see that the set $\Delta_{y'}^k(E)$ is the image of $\Pi_y^k(E)$ via $\Gamma_y$, and vice versa, i.e.
\begin{align*}
    \Delta_{y'}^k(E')=\Gamma_y\big(\Pi_y^k(E)\big)\quad \text{ and }\quad \Pi_y^k(E)=\Gamma_y^{-1}\big(\Delta_{y'}^k(E')\big).
\end{align*}
Moreover, one can see that $\Gamma_y$ is a homeomorphism, and hence bi-Lipschitz over compact sets. This implies that
\begin{align}\label{eq_equal_dim}
    \dim_H(\Delta_{y'}^k(E'))=\dim_H(\Pi_y^k(E)).
\end{align}
In terms of Lebesgue measures, one also has
\begin{align}\label{eq_equivalent_positive_L_measures}
    \mathcal{L}^k\big(\Delta_{y'}^k(E')\big)>0 \quad \Longleftrightarrow\quad  \mathcal{L}^k\big(\Pi_y^k(E)\big)>0.
\end{align}
%Roughly speaking, in order to prove that $\Pi_y^k(E)$ has positive Lebesgue measure, it suffices to show that $\Delta_{y'}^k(E')$ has positive Lebesgue measure, and vice versa. 
Furthermore, Lemma \ref{lem_equivalent_finitness_L2norm} below provides us stronger information: an equivalent between $L^2$ norm of natural measures supported on $\Pi_y^k(E)$ and $\Delta_{y'}^k(E')$. 

Assume that $\mu_E$ and $\mu_F$ are corresponding Frostman probability measures supported on $E$ and $F$, with exponents $s_E<\dim_H(E)$ and $s_F<\dim_H(F)$, respectively. Then we may define measures which supported on $E'$ and $F'$, by
\begin{align}\label{eq_pushforward_measures}
    \mu_{E'}:=H_{1\sharp}\mu_E\quad,\quad \mu_{F'}:=H_{2\sharp}\mu_F.
\end{align}
The bi-Lipschitz property of $H_1$ and $H_2$ ensures that $\mu_{E'}$ and $\mu_{F'}$ are Frostman probability measures with exponents $s_E$ and $s_F$, respectively.

With these notations in hand, let us define a natural measure $\eta_{y'}$, supported on $\Delta_{y'}^k(E')$, by the following relation
\begin{align}\label{eq_definition_eta}
    \int_{\R^k} f(t)d\eta_{y'}(t)=\int_{E'}f(\,\vert y^{1'}-\overline{x}\vert,\ldots, \vert y^{k'}-\overline{x}\vert\,) d\mu_{E'}(\overline{x})\,,\quad \forall\, f\in C_0(\R^k).
\end{align}
Then we may define measure $\nu_y$, supported on $\Pi_y^k(E)$, by relation
\begin{align}\label{eq_definition_nu}
    \int_{\R^k}g(u)d\nu_y(u) = \int_{\R^k}g(\,\Gamma_y^{-1}(t)\,)d\eta_{y'}(t), \quad \forall\, g\in C_0(\R^k).
\end{align}
%Using simple calculation, one can check that the measure $\nu_y$ defined as in \eqref{eq_definition_nu} is actually a natural measure supported on $\Pi_y^k(E)$, i.e. 
%\begin{align}\label{eq_definition_nu_natural}
%    \int g(u)d\nu_y(u) =\int g\big( y^1\cdot x,\dots, y^k\cdot x\big)\, d\mu_E(x), \quad \forall\, g\in C_0(\R^k).
%\end{align}
%To avoid repetition, we will use the above definitions of measures $\eta$ and $\nu$ for the rest of the paper.
%Explain
%\begin{align*}
%    \int g(u)d\nu_y(u)
%    &=\int g\big(\Gamma_y^{-1}(t)\big) \, d\eta_{y'}(y)\\
%    &=\int g\big(\Gamma_y^{-1}(\, \vert y^{1'}-\overline{x}\vert,\ldots, \vert y^{k'}-\overline{x}\vert \,)\big)\, d\mu_{E'}(\overline{x})\\
%    &=\int g\big(y^1\cdot H_1^{-1}(x),\ldots, y^k\cdot H_1^{-1}(x)\big)\, d(H_{1\sharp}\mu_E)(\overline{x})\\
%    &=\int g\big( y^1\cdot x,\dots, y^k\cdot x\big)\, d\mu_E(x).
%\end{align*}

\begin{lemma}\label{lem_equivalent_finitness_L2norm}
With the notations above, for each $y=(y^1,\dots, y^k)\in F^k$, we have
\begin{align}
    \Vert \nu_y\Vert_{L^2}\approx \Vert \eta_{y'}\Vert_{L^2}.
\end{align}
\end{lemma}

\begin{proof}
We will show that
$
    \Vert \nu_y\Vert_{L^2}\lesssim \Vert \eta_{y'}\Vert_{L^2}.
$
The opposite inequality can be obtained by the same argument.

By duality, it enough to show that
\begin{align*}
    \langle g, \nu_y\rangle \lesssim \Vert g\Vert_{L^2}\cdot \Vert \eta_{y'}\Vert_{L^2}.
\end{align*}
From the definition of $\nu_y$, the left-hand side equals
\begin{align*}
    \int g(u)d\nu_y(u) 
    =\int g(\Gamma^{-1}(t))d\eta_{y'}(t).
\end{align*}
Apply the Cauchy-Schwarz inequality, and using the fact that $\Gamma_y^{-1}$ has a bounded Jacobian, one gets
\begin{align*}
    \langle g, \nu_y\rangle 
    &\leq \bigg(\int \vert (g\circ \Gamma^{-1})(t)\vert^2dt\bigg)^{1/2}\bigg(\int \vert \eta_{y'}(t)\vert^2dt\bigg)^{1/2}
    \lesssim \Vert g\circ \Gamma^{-1}\Vert_{L^2}\cdot \Vert \eta_{y'}\Vert_{L^2}
    \lesssim \Vert g\Vert_{L^2}\cdot \Vert \eta_{y'}\Vert_{L^2}.
\end{align*}
The proof of lemma is completed.
\end{proof}

\section{Proofs of Theorems \ref{thm1.1}, \ref{thm1.2}, and \ref{thm1.3}}\label{sec_proof123}

In this section, we will apply the connection which is established in the previous section to give the proofs of Theorems \ref{thm1.1}, \ref{thm1.2}, and \ref{thm1.3}. 
\subsection{Proof of Theorem \ref{thm1.1}}
Theorem \ref{thm1.1} is an immediate consequence of the following.
\begin{theorem}\label{thm_1a}
Let $E,F$ be a pair of compact sets on $P_d^\ast$, $d\geq 3$. Assume that $\mu_E$ and $\mu_F$ are Frostman probability measures supported on $E$ and $F$ respectively, with corresponding exponents $s_{\mu_E}<\dim_H(E)$, $s_{\mu_F}<\dim_H(F)$. Assume that
\begin{align}\label{eq_condition_measures}
   s_{\mu_E}+\beta(s_{\mu_F})>d-1.
\end{align}
Then for $\mu_F$ a.e. $y\in F$,
\begin{align}\label{eq_finiteness_L2norm_nu}
    \int\vert \nu_y(t)\vert^2dt<\infty.
\end{align}
In particular, for $\mu_F$ a.e. $y\in F$, $\mathcal{L}^1(\Pi_y^1(E))>0$. Here $\nu_y$ is the natural measure supported on $\Pi_y^1(E)$, defined in \eqref{eq_definition_nu}.
\end{theorem}

\begin{proof}[Proof of Theorem \ref{thm1.1}]
Apply Theorem \ref{thm_1a} with $E=F$, and $s_{\mu_E}=s_{\mu_F}$. Solving \eqref{eq_condition_measures} will give us $s_{\mu_E}>\frac{d-1}{2}+\frac{1}{4}+\frac{1}{8d-12}$ as required. The existence of $y\in E$ such that $\Vert \nu_y\Vert_{L^2}<\infty$ also follows.
\end{proof}

\begin{proof}[Proof of Theorem \ref{thm_1a}]
To prove this theorem, it suffices to show that \eqref{eq_finiteness_L2norm_nu} holds for $\mu_F$ a.e. $y\in F$. Then in fact $d\nu_y$ is absolutely continuous with respect to Lebesgue measure, with density in $L^2$, also denoted by $\nu_y(t)$. Apply the Cauchy-Schwarz inequality, one has
\begin{align*}
    1=\bigg(\int d\nu_y(t)\bigg)^2 \leq \bigg(\int_{\Pi_y^1(E)}dt\bigg)\cdot \bigg(\int\vert \nu_y(t)\vert^2dt\bigg)
\end{align*}
which yields that $\mathcal{L}^1(\Pi_y^1(E))>0$.

In order to prove \eqref{eq_finiteness_L2norm_nu}, we make use of the notations from Section \ref{sec_connection} with $n=d$ and $k=1$. In this context, the pinned distance set determined by $E'$ at $y'\in F'$ is defined by
\begin{align*}
    \Delta_{y'}(E'):=\{\, \vert y'-\overline{x}\vert : \overline{x}\in E'\,\}.
\end{align*}
Denote $\eta_{y'}$ as the natural measure on $\Delta_{y'}(E')$, which is defined in \eqref{eq_definition_eta}. Then, apply Proposition \ref{prop_Liu_pinned_distance} with compact sets $E'$ and $F'$, one can verify that
\begin{align*}
    \iint\vert \eta_{y'}(t)\vert^2dtd\mu_{F'}(y')<\infty.
\end{align*}
Due to Lemma \ref{lem_equivalent_finitness_L2norm}, one has for $\mu_F$ a.e. $y\in F$,
\begin{align}\label{eq_estimate_L2_pushforward}
    \int \vert \nu_{y}(u)\vert^2du\,\lesssim \int\vert \eta_{y'}(t)\vert^2dt\,<\,\infty.
\end{align}
Thus \eqref{eq_finiteness_L2norm_nu} holds and we complete the proof of Theorem \ref{thm_1a}.
\end{proof}

\subsection{Proof of Theorem \ref{thm1.2}}
Theorem \ref{thm1.2} follows by the following.
\begin{theorem}\label{thm_EF_P3}
Let $E, F\subset P_3^{\ast}$ be a pair of compact sets with $\dim_H(E)$, $\dim_H(F)>\frac{5}{4}$.
Suppose that $\mu_E$ and $\mu_F$ are $s$-Frostman probability measures on $E$ and $F$, respectively, with $\dim_H(E),\dim_H(F)>s>\frac{5}{4}$. Then 
\begin{align*}
    \mu_F\big(\{ y\in F: \mathcal{L}^1(\Pi_y^1(E))>0\}\big)>0.
\end{align*}
\end{theorem}

\begin{proof}
As in the proof of Theorem \ref{thm_1a}, we make use of the notations from section \ref{sec_connection} with $n=3$ and $k=1$.

%By assumption, the sets $E'=H_1(E)$ and $F'=H_2(F)$ satisfy
%\begin{align*}
%    \dim_H(E')=\dim_H(E)\,\,,\,\, \dim_H(F')=\dim_H(E).
%\end{align*}
%Then the corresponding measures $\mu_{E'}, \mu_{F'}$ supported on these sets are $s$-Frostman probability measures. 
Apply Proposition \ref{prop_outaylor} with the pair of sets $E'$, $F'\subset \R^2$, we have
\begin{align*}
    \mu_{F'}(\{ y'\in F': \mathcal{L}^1(\Delta_{y'}(E))>0\})>0,
\end{align*}
where $\mu_{F'}=H_{2\sharp}\mu_F$. Then by definition of pushforward measure, and the equivalence \eqref{eq_equivalent_positive_L_measures}, we conclude that
\begin{align*}
    \mu_F\big(\{ y\in F: \mathcal{L}^1(\Pi_y^1(E))>0\}\big)>0.
\end{align*}
The proof of Theorem \ref{thm_EF_P3} is thus completed.
\end{proof}

\subsection{Proof of Theorem \ref{thm1.3}}
\begin{proof}[Proof of Theorem \ref{thm1.3}]
As in proof of Theorem \ref{thm_1a}, we use the notations from section \ref{sec_connection} with $n=d$ and $k=1$, but here we assume that $E=F$, $\dim_H(E)=\dim_H(F)=\dim_H(E')=\dim_H(F')=s.$

\begin{itemize}
    \item[(i)] Apply Proposition \ref{prop_IosevichLiu19} $(i)$ with compact sets $E',F'\subset \R^{d-1}$. In view of Proposition \ref{prop_IosevichLiu19} $(i)$, if
    \begin{align}\label{eq_dim_beta_s}
        \dim_H(E')+\frac{d-2}{d}\dim_H(F')>d-2+\beta.
    \end{align}
    then there exists a Frostman measure $\mu_{F'}$ supported on $F'$ such that for $\mu_F'$ a.e. $y'\in F'$,
    \begin{align}\label{eq_dim_beta}
       \dim_H(\Delta_{y'}(E'))\geq \beta.
    \end{align}
    Solving \eqref{eq_dim_beta_s}, one has $s>\frac{d^2-2d+\beta d}{2(d-1)}$. Combining \eqref{eq_dim_beta} with equality \eqref{eq_equal_dim}, we have for $\mu_F$ a.e. $y\in F$,
    \begin{align*}
        \dim_H(\Pi_y^1(E))=\dim_H(\Delta_{y'}(E'))\geq \beta.
    \end{align*}
    \item[(ii)] Arguing as above, one can apply Proposition \ref{prop_IosevichLiu19} $(iii)$ to get the result.
\end{itemize}
\end{proof}

\subsection{Pinned dot-product trees}
In this subsection, we discuss an extension of Theorems \ref{thm1.1} and \ref{thm1.2} to pinned dot-product trees.

Let $T$ be an arbitrary tree with $k+1$ vertices and $k$ edges. Assume that $V(T)=\{v_1, \ldots, v_{k+1}\}$, then the edge set of $T$ can be enumerated as follows:
\[\mathcal{E}(T)=\{ (v_{i_{1}}, v_{i_{2}}), (v_{i_3}, v_{i_4}), \ldots, (v_{i_{2k-1}}, v_{i_{2k}})\},\]
where $i_1\le i_3\le \cdots\le i_{2k-1}$, and $i_{2s}< i_{2t}$ if $s<t$ and $i_{2s-1}=i_{2t-1}$. 

The edge-length vector of $T$ is defined by 
\[(v_{i_1}\cdot v_{i_2}, v_{i_3}\cdot v_{i_4}, \ldots, v_{i_{2k-1}}\cdot v_{i_{2k}}) \in \mathbb{R}^k.\]
Given a set $A\subset \mathbb{R}^2$, we say a tree $T'$, with vertices in $A$ given by $V(T')=\{x_1, \ldots, x_{k+1}\}$, is isomorphic to $T$ if there is a map $\varphi\colon V(T')\to V(T)$ such that $(x, y)$ is an edge of $T'$ iff $\left(\varphi(x), \varphi(y)\right)$ is an edge of $T$. For $x\in V(T')$ and $v\in V(T)$, we say $(T', x)$ is isomorphic to $(T, v)$ if $T$ is isomorphic to $T'$ and $\varphi(x)=v$. We say two isomorphic trees $T$ and $T'$ with $k$ edges are distinct if their edge-length vectors are distinct.

Combining Theorems \ref{thm1.1} and \ref{thm1.2} with an analogue argument in \cite{OuTaylor}, one can able to obtain the following result for pinned dot-product trees. 

\begin{theorem}
\label{thm_pinned_trees}
Given a tree $T$ with $k\ge 1$ edges and an arbitrary vertex $v$. Let $E$ be a compact set on $P_d$, $d\geq 3$. If $\dim_H(E)>\alpha(d)$, then there exists $x\in E$ such that the set of edge-lengths of trees $(T', x)$ which are isomorphic to $(T, v)$ has positive Lebesgue measure. \\
Here $\alpha(2)=\frac{5}{4}$, and $\alpha(d)=\frac{d-1}{2}+\frac{1}{4}+\frac{1}{8d-12}$, for $d\geq 3$.
\end{theorem}

\section{Proof of Theorem \ref{thm_distance_simplex}}\label{sec_proof5}
In this section, we will apply the connection in Section \ref{sec_connection} to derive a proof for the pinned simplex problem. We first prove several technical lemmas.
\begin{lemma}\label{lem_dotproduct}
Let $E,F\subset P_n^\ast$ be a pair of compact sets with $\dim_H(E)=\dim_H(F)>s>\frac{n+k-1}{2}$, $1\leq k\leq n$. Assume that $\mu_{E}$, $\mu_{F}$ are $s$-Frostman measure supported on $E$ and $F$, respectively. Then we have
\begin{align*}
    \int_{F^k} \int_{\R^k}\vert \nu_{y}(u)\vert^2\,du \, d\mu_{F}^k(y) \,\,<\,\,\infty,
\end{align*}
where for each $y=(y^1,\dots, y^k)\in F$, we denote $\nu_{y}$ as the natural measure supported on $\Pi_{y}^k(E)$
%, see \eqref{eq_definition_nu_natural}
.
\end{lemma}
The proof of Lemma \ref{lem_dotproduct} is the same as in \cite[Theorem 4]{EHI}, but we present it here for completeness.

\begin{proof}
First, we may assume that
\begin{align*}
    0<c\leq \vert x\vert, \vert y\vert\leq 1, \quad \forall\, x\in E, \, y\in F, \quad \text{ for some positive constant $c$}.
\end{align*}
Let $y=(y^1,\dots, y^k)\in F^k$, by a simple calculation, one sees that 
\begin{align*}
    \widehat{\nu_y}(t)
   % &=\int e^{-2\pi it\cdot u}d\nu_y(u)=\int e^{-2\pi i t\cdot (x\cdot y^1,\dots, x\cdot y^k)}d\mu_E(x)\\
   % &=\int e^{-2\pi i (t\cdot y)\cdot x}d\mu_E(x)\\
    &=\widehat{\mu_E}(t\cdot y),\quad \text{ for all $t=(t_1,\dots, t_k)\in \R^k$},
\end{align*}
where $t\cdot y=t_1y^1+\cdots +t_ky^k$.

It follows that
\begin{align*}
    \iint \vert \widehat{\nu_y}(t)\vert^2\,dt\,d\mu_F^k(y)=\iint \vert \widehat{\mu_E}(t\cdot y)\vert^2dtd\mu_F^k(y).
\end{align*}
Let $\phi\in C_0^\infty(\R^n)$ be a smooth cut-off function satisfying $\chi_{B(0,m/2)}\leq \phi\leq \chi_{B(0,2)}$. Then the above quantity is dominated by
\begin{align*}
%    \iint \vert \widehat{\nu_y}(t)\vert^2\,dt\,d\mu_F^k(y)
%    &=\iint \vert \widehat{\mu_E}(t\cdot y)\vert^2dtd\mu_F^k(y)\\
    &\lesssim \iint \vert \widehat{\mu_E}\ast\widehat{\phi}(t\cdot y)\vert^2dtd\mu_F^k(y)
    \lesssim \int \vert\widehat{\mu_E}(\xi)\vert^2\,\bigg(\iint \vert \widehat{\phi}(ty-\xi)\vert dtd\mu_F^k(y)\bigg)\,d\xi.
\end{align*}
Using the same argument as in \cite[Theorem 4]{EHI}, one can check that
\begin{align*}
    \iint \vert \widehat{\phi}(ty-\xi)\vert dtd\mu_F^k(y)\lesssim \vert \xi\vert^{-s+k-1}.
\end{align*}
Hence, from the assumption that $s>\frac{n+k-1}{2}$, we get
\begin{align*}
   \iint \vert \widehat{\nu_y}(t)\vert^2\,dt\,d\mu_F^k(y) 
   \lesssim \int \vert\widehat{\mu_E}(\xi)\vert^2 \vert \xi\vert^{-s+k-1}\,d\xi \,\lesssim \, \int \vert\widehat{\mu_E}(\xi)\vert^2 \vert \xi\vert^{s-n}\,d\xi \, <\infty.
\end{align*}
\end{proof}
Using $n=d+1$, we obtain the following.
\begin{lemma}\label{lem_kpinned_distances}
Assume that $E',F'\subset \R^d$ are compact sets, with $\dim_H(E'),\dim_H(F')>s>\frac{d+k}{2}$, for $1\leq k\leq d$. 

Then there exist $s$-Frostman probability measures supported on $E'$ and $F'$ respectively, such that for $\mu_{F'}^k$ a.e. $y'=(y^{1'},\dots, y^{k'})\in (F')^{k}$, 
\begin{align*}
    \int \vert \eta_{y'}(t)\vert^2dt<\infty,
\end{align*}
where $\eta_{y'}$ is the natural measure supported on $\Delta_{y'}^k(E')$, see \eqref{eq_definition_eta}.
\end{lemma}

\begin{proof}
%Assume that $E',F'\subset \R^d$ are compact sets satisfying the assumption of the lemma. Let $\mu_{E'}$ and $\mu_{F'}$ be $s-$Frostman probability measures supported on $E'$ and $F'$. By pigeonhole and then replacing $E'$ and $F'$ with their subsets, we can assume that $E'$ and $F'$ are contained in two disjoint balls with positive distance. Next, restrict the measures on the corresponding subsets, and multiply by suitable positive constants, we may get $s-$Frostman probability measures supported on each of these balls, also denoted by $\mu_{E'}$ and $\mu_{F'}$, respectively.

First, by pigeon-hole, we can assume that $E'$ and $F'$ are contained in two disjoint balls with positive distance. Moreover, since the $k-$simplex problem is invariant under rotations, translations, and dilations, we can further assume that
\begin{align*}
    E'\subset B(0,1)\setminus B(0,c)\,\,\text{ and }\,\,F'\subset \,\R^d \setminus B(0,1),
\end{align*}
for some constant $c>0$, where $B(0,r)$ denotes the open ball centered at $0$ of radius $r>0$.
%where $\R_{>}^d:=\{ (x_1,\dots, x_d)\in \R^d \colon x_i>0, \,\forall \,1\leq i\leq d\,\}$, and $\R_{<}^d:= -\R_{>}^d$.
Recall that in Section \ref{sec_connection} with $n=d+1$, we consider the following maps:
\begin{align*}
    H_1: P_{d+1}^\ast \to (\R^\ast)^d \quad, \quad H_2:P_{d+1}^\ast\to (\R^\ast)^d.
\end{align*}
Put $\tilde{E}=H_1^{-1}(E)$, $\tilde{F}=H_2^{-1}(F')$. Notice that these sets are indeed contained on the truncated paraboloid $P_{d+1}$.

Now let $\mu_{E'}$ and $\mu_{F'}$ be $s-$Frostman probability measures supported on $E'$ and $F'$, respectively. We define the corresponding measures supported on $\tilde{E}$ and $\tilde{F}$ by
\begin{align*}
    \mu_{\tilde{E}}=H_{1\sharp}^{-1}(\mu_{E'}),\quad \mu_{\tilde{F}}=H_{2\sharp}^{-1}(\mu_{F'}).
\end{align*}
We see that $\tilde{E}$ and $\tilde{F}$ satisfy conditions of Lemma \ref{lem_dotproduct}. Hence, Lemma \ref{lem_dotproduct} implies that for $\mu_{\tilde{F}}^k$ a.e. $\tilde{y}=(\tilde{y}^1,\dots, \tilde{y}^k)\in \tilde{F}^k$, we have
\begin{align*}
    \int \vert \nu_{\tilde{y}}(u)\vert^2du  \,<\,\infty,
\end{align*}
where $\nu_{\tilde{y}}$ is the natural measure supported on $\Pi_{\tilde{y}}^k(\tilde{E})$.

Apply Lemma \ref{lem_equivalent_finitness_L2norm}, we conclude that for $\mu_{F'}^k$ a.e. $y'=(y^{1'},\dots, y^{k'})\in (F')^k$,
\begin{align*}
    \int \vert \eta_{y'}(t)\vert^2\,dt \lesssim \int \vert \nu_{\tilde{y}}(u)\vert^2du  \,<\,\infty,
\end{align*}
as required. The proof of lemma is completed.
\end{proof}

Now we are ready to give the proof of Theorem \ref{thm_distance_simplex}.
\begin{proof}[Proof of Theorem \ref{thm_distance_simplex}]
Let $E\subset \R^d$ be a compact set, with $\dim_H(E)>\frac{d+k}{2}$, $1\leq k\leq d$. Assume that $\mu$ is an $s-$Frostman probability measure supported on $E$, with $\dim_H(E)>s>\frac{d+k}{2}$. To avoid repetition, for any set $A\subset E$, we denote $\mu_A=\mu\vert_A$ as the restriction of measure $\mu$ on $A$.

We will show that there exist subsets $E_1,\ldots, E_{k+1}$ of $E$, such that $\mu(E_l)>0$, for all $ \, 1\leq l\leq k+1$, and for positive many $y^1\in E_1$, the set
\begin{align*}
 T_{y^1}(E_2,\dots,E_{k+1}):=\Big\{  \Big( \big\vert y^i-y^j\big\vert\Big)_{1\leq i<j\leq k+1}\colon y^j\in E_j, \forall\, 2\leq j\leq k+1\Big\}
\end{align*}
has positive Lebesgue measure.

In fact, we will prove the following.

\begin{claim}\label{claim_conditionEF}
There exist subsets $E_l$, $F_l$ of $E$, for $2\leq l\leq k+1$, such that:
\begin{itemize}
    \item[(i)] $E_{k+1}, F_{k+1}\subset E$\,; and  $E_l, F_l\subset F_{l+1}$, for $2\leq l\leq k$.
    \item[(ii)] For each $2\leq l\leq k+1$, $E_l$ and $F_l$ are contained in two disjoint balls of positive distance, and $\mu(E_l)>0$, $\mu(F_l)>0$.
    \item[(iii)] For each $2\leq l\leq k+1$, we have for $\mu_{F_l}^{l-1}$ a.e. $(y^1,\dots, y^{l-1})\in F_l^{l-1}$, 
    \begin{align*}
        T_{y^1,\dots, y^{l-1}}(E_{l}, \dots, E_{k+1})
        :=\Big\{ \Big( \big\vert y^i-y^j\big\vert\Big)_{\substack{1\leq i<j\leq k+1\\ j\geq l}}\colon y^j\in E_j, \forall\, l\leq j\leq k+1\Big\}.
    \end{align*}
    has positive $\mathcal{L}^{k+\cdots + (l-1)}$ measure.
\end{itemize}
\end{claim}

\begin{proof}

We will prove the claim by induction and begin with case $l=k+1$.

\textbf{Case $l=k+1$:} \\
Consider the set $E$, by pigeon-hole, there exist two subsets of $E$, denoted by $E_{k+1}$ and $F_{k+1}$, such that they are contained in two disjoint balls with positive distance, and
$
    \mu(E_{k+1})>0\,,\,\, \mu(F_{k+1})>0.
$
%Denote $\mu_{E_{k+1}}=\mu\vert_{E_{k+1}}$ and $\mu_{F_{k+1}}=\mu\vert_{F_{k+1}}$ as measures on $E_{k+1}$ and $F_{k+1}$, respectively.

In view of Lemma \ref{lem_kpinned_distances}, it is not difficult to see that for $\mu_{F_{k+1}}^k$ a.e. $y=(y^1,\dots, y^k)\in F_{k+1}^k$, 
\begin{align}
    \int \vert \eta_y(t)\vert^2\,dt<\infty, \quad \text{ and } \quad \mathcal{L}^k\big(\Delta_{y}^k(E_{k+1})\big)>0,
\end{align}
where $\Delta_y^k(E_{k+1})=\big\{ (\vert y^1-x\vert, \vert y^2-x\vert, \dots, \vert y^k-x\vert\,)\colon\, x\in E_{k+1}\big\}$. Notice that $\Delta_y^k(E_{k+1})$ coincides with $T_{y^1,\dots,y^{k}}(E_{k+1})$, so the case $l=k+1$ holds true.

\textbf{Inductive step:} \\
Assume that for $l+1\leq j\leq k+1$, we constructed disjoint sets $E_{j}, F_{j}$ satisfying $(i), (ii)$ and $(iii)$. In particular, assume that for $\mu_{F_{l+1}}^l$ a.e. $y=(y^1,\dots, y^l)\in F_{l+1}^{l}$,
\begin{align}\label{eq_positive_induction}
    \mathcal{L}^{k+\cdots+l}\Big(T_{y^1,\dots, y^{l}}\big(E_{l+1},\dots, E_{k+1}\big)\Big)\,>\,0.
\end{align}
We need to prove that there exist $E_{l}, F_{l}$ satisfy conditions $(i),(ii),(iii)$ of the claim.

Using pigeon hole again with compact set $F_{l+1}$, there exist two subsets $E_{l}$, $F_{l}\subset F_{l+1}$ such that they are contained in two disjoint balls with positive distance, 
$    \mu(F_{l})>0$ and $\mu(E_{l})>0.$
Then \eqref{eq_positive_induction} yields that for $\mu_{F_{l}}^{l-1}\times \mu_{E_{l}}$ a.e. $y=(y^1,\dots, y^{l-1},y^l)\in F_{l}^{l-1}\times E_{l}$,
\begin{align}\label{eq_positive_kdistances_restriction}
    \mathcal{L}^{k+\cdots+l}\Big(T_{y^1,\dots, y^{l}}\big(E_{l+1},\dots, E_{k+1}\big)\Big)\,>\,0.
\end{align}
Put
$
    \mathcal{E}_l:=\big\{ (y^1,\dots, y^{l-1},y^l)\in F_{l}^{l-1}\times E_{l}\,\colon \mathcal{L}^{k+\cdots+l}\Big(T_{y^1,\dots, y^{l}}(E_{l+1},\dots, E_{k+1})\Big)\,>\,0\big\}.
$\\
The above discussion implies that 
\begin{align}\label{eq_equal_measures}
    \mu_{F_{l}}^{l-1}\times \mu_{E_{l}}\big(\mathcal{E}_l\big) =\big(\mu(F_{l})\big)^{l-1} \cdot \mu(E_{l})\, >\,0.
\end{align}
Now, for each $y'=(y^1,\dots, y^{l-1})\in F_{l}^{l-1}$, we put
\begin{align*}
    \mathcal{E}_{l,y'}:=\big\{ y^l\in E_{l}\colon (y^1,\dots, y^{l-1},y^l)\in \mathcal{E}_{l}\big\}.
\end{align*}
Then define
$
    \mathcal{F}_{l-1}:=\big\{ (y^1,\dots, y^{l-1})\in F_{l}^{l-1}: \mu_{E_{l}}\big(\mathcal{E}_{l,y'}\big)=\mu(E_{l})\big\}.
$
Using Fubini and \eqref{eq_equal_measures}, it is easy to see that
\begin{align}\label{eq_equal_measures_F}
    \mu_{F_{l}}^{l-1}\big(\mathcal{F}_{l-1}\big)=\big(\mu(F_{l})\big)^{l-1}>0.
\end{align}
Next, for each $y'=(y^{1},\dots, y^{l-1})\in F_{l}^{l-1}$, we define the following sets
\begin{align*}
 \Delta_{y'}^{l-1}\big(E_{l}\big)&:= \big\{ (\vert z-y^1\vert,\dots, \vert z-y^{l-1}\vert)\colon z\in E_{l}\big\},\\
    \Delta_{y'}^{l-1}\big(\mathcal{E}_{l,y'}\big)&:= \big\{ (\vert z-y^1\vert,\dots, \vert z-y^{l-1}\vert)\colon z\in \mathcal{E}_{l,y'}\big\}.
\end{align*}
Let $\eta_{y'}$ be the natural measure supported on $\Delta_{y'}^{l-1}\big(E_{l}\big)$. Then we may define measure $\theta_{y'}$, supported on $\Delta_{y'}^{l-1}\big(\mathcal{E}_{l,y'}\big)$, by relation
\begin{align*}
    \int_{\R^{l-1}}f(t)d\theta_{y'}(t)=\int_{\R^{l-1}}f(t)\cdot \chi_{ \Delta_{y'}^{l-1}(E_{l})}(t)\, d\eta_{y'}(t),\quad \forall\, f\in C_0(\R^{l-1}).
\end{align*}
Using the fact that $\mu_{E_{l}}\big(E_{l-1}\setminus \mathcal{E}_{l,y'}\big)=0$, one can check that for all $\xi\in \R^{l-1}$,
\begin{align*}
    \widehat{\theta_{y'}}(\xi)
    %&=\int e^{-2\pi i  \xi\cdot t}\chi_{\Delta_{y'}^{l-1}\Big(E_{l-1}\Big)}(t)d\eta_{y'}(t)\\
    %&=\int_{\mathcal{E}_{l,y'}} e^{-2\pi i \xi\cdot(\vert x-y^1\vert,\dots, \vert x-y^{l-1}\vert}d\mu_{E_{l-1}}(x)\\
    %&=\int_{E_{l-1}} e^{-2\pi i \xi\cdot(\vert x-y^1\vert,\dots, \vert x-y^{l-1}\vert)}d\mu_{E_{l-1}}(x)\\
    %&=\int e^{-2\pi i \xi\cdot t}\, d\eta_{y'}(t)\\
    &=\widehat{\eta_{y'}}(\xi).
\end{align*}
Hence by Plancherel formula, one has
\begin{align}\label{eq_equal_L2norm_simplex}
    \int \vert \theta_{y'}(t)\vert^2dt = \int \vert \eta_{y'}(t)\vert^2dt.
\end{align}
On the other hand, apply Lemma \ref{lem_kpinned_distances} for the pair of sets $F_{l}$ and $E_{l}$, it is not difficult to see that for $\mu_{F_{l}}^{l-1}$ a.e. $y'\in F_{l}^{l-1}$,
\begin{align*}
    \int \vert \eta_{y'}(t)\vert^2dt<\infty.
\end{align*}
Combining with \eqref{eq_equal_measures_F} and \eqref{eq_equal_L2norm_simplex}, we conclude that for $\mu_{F_{l}}^{l-1}$ a.e. $y'=(y^1,\dots, y^{l-1})\in F_{l}^{l-1}$,
\begin{align*}
    \mathcal{L}^{l-1}\big(\Delta_{y'}^{l-1}\big(\mathcal{E}_{l,y'}\big)\big)>0.
\end{align*}
According to \eqref{eq_positive_kdistances_restriction}, the above implies that for $\mu_{F_{l}}^{l-1}$ a.e. $y'=(y^1,\dots, y^{l-1})\in F_{l}^{l-1}$,
\begin{align*}
    \mathcal{L}^{k+\cdots+(l-1)}\Big(T_{y^1,\dots, y^{l-1}}(E_{l},\dots, E_{k+1})\Big)\,>\,0.
\end{align*}
The induction step then follows, and we complete the proof of claim.
\end{proof}

From Claim \ref{claim_conditionEF}, put $E_1=F_2$, we have for $\mu_{E_1}$ a.e. $y^1\in E_1$,
\begin{align*}
    \mathcal{L}^{\binom{k+1}{2}}\Big(T_{y^1}(E_2,\dots, E_{k+1})\Big)\,>\,0.
\end{align*}
Thus we complete the proof of Theorem \ref{thm_distance_simplex}.
\end{proof}

\section{Sharpness of Theorem \ref{thm1.1}}\label{sharpness}
In this section, we will give an example to show that $\frac{d-1}{2}$ is the best dimensional threshold that one can expect for the problem of product of sets on paraboloids. This example is inspired by a construction due to K. Falconer in $\cite{Fal85}$.

First, for convenience, in this section, we denote the truncated paraboloid in $\R^d$ by $P_d:=\{(\overline{x}, \vert \overline{x}\vert^2)\in \R^d: |\overline{x}|\le 3\}$.
\begin{proposition}
Let $d\geq 3$. For any $0<s\le d$, there exists a compact set $E\subset P_d$ with $\dim_H(E)=s$ and
\begin{align*}
    \dim_H\left(\Pi(E)\right)\leq \frac{2}{d-1}\dim_H(E).
\end{align*}
\end{proposition}

\begin{proof}
Fix $0<s\leq d$. For each $x\in \R^d$, we denote $\overline{x}=(x_1,\dots, x_{d-1})\in \R^{d-1}$. Choose a rapidly increasing sequence of positive integers $\{q_k\}_{k=1}^\infty$, say $q_{k+1}\geq q_k^k$ for each $k$.

For each $k$, we define
\begin{align*}
    E_k'&:=\{ \overline{x}\in \R^{d-1}: 1\leq x_j\leq 2,\, \vert x_j-p_j/q_k\vert \leq q_k^{-(d-1)/s},\\ 
    &\hspace{2cm}\text{ for some integer } q_k\leq p_j\leq 2q_k, \text{ for } j=1,2,\dots, d-1\}.
\end{align*}
Then put $E'=\bigcap_{k=1}^\infty E_k'$. By elementary calculations, it can be checked that $\dim_H(E')=s$, see \cite[Theorem 8.15]{Falbook85} or \cite[Example 4.7]{Falbook14}.

Define $
    E= \{ (\overline{x},\vert \overline{x}\vert^2)\in P_d \colon \overline{x}\in E'\}$.
    One sees that $E$ lies on the paraboloid $P_d$ and $\dim_H(E)=s$ as required. Thus we only need to check that 
    \begin{align*}
        \dim_H(\Pi(E))\leq 2\dim_H(E)/(d-1).
    \end{align*}

Consider the map 
\begin{align*}
    \Phi: [1,2]^{d-1}\times [1,2]^{d-1}\to \big[ 0,\infty\big)\,\, , \quad (\overline{x},\overline{y})\mapsto \bigg\vert \frac{\overline{y}}{\vert \overline{y}\vert}+2\vert \overline{y}\vert \overline{x}\bigg\vert^2.
\end{align*}
For all $x, y\in E$, we have the following equivalence
\begin{align*}
    x\cdot y=t \quad \Longleftrightarrow \quad
    \Phi(\overline{x},\overline{y})=1+4t.
\end{align*}
In particular, this implies that $\dim_H(\Pi(E))=\dim_H(\Phi(E',E'))$. So it suffices to give an upper bound for $\dim_H(\Phi(E',E'))$.

Clearly, one has $\Phi(E',E')\subset \bigcap_{k=1}^\infty \Phi(E_k',E_k')$, so any cover of $\Phi(E_k',E_k')$ would also a cover $\Phi(E',E')$ as well. From definition, the set $E_k'$ consist of cubes of side length $q_k^{-(d-1)/s}$, whose centered at points of the form $(p_1/q_k,\dots, p_{d-1}/q_k)$. Thus given any $\overline{x}, \overline{y}\in E_k'$, we must have $\overline{x}$ and $\overline{y}$ are contained in some cubes centered at $p_{\overline{x}}$ and $p_{\overline{y}}$, respectively. Moreover, one can check that the map $\Phi$ has continuous derivative, which is bounded on compact set. This yields that $\Phi$ is Lipschitz with some Lipschitz constant $L>0$. Combining these facts, one has
\begin{align*}
    \big\vert\Phi(\overline{x},\overline{y}) -\Phi(p_{\overline{x}},p_{\overline{y}}) \big\vert \lesssim Lq_k^{-(d-1)/s}, \quad \forall \overline{x},\overline{y}\in E_k'.
\end{align*}

Let $p=(p_1/q_k,\dots, p_{d-1}/q_k)$ and $p'=(p_1'/q_k,\dots, p_{d-1}'/q_k)$, where  $q_k\leq p_j,p_j'\leq 2q_k$. Using the fact that $\sqrt{d}\leq \vert p'\vert\leq 2\sqrt{d}$, we have
\begin{align*}
    \Phi(p,p')=\bigg\vert \frac{p'}{\vert p'\vert
    }+2\vert p'\vert p\bigg\vert^2 \leq  \frac{r}{q_k^2},
\end{align*}
for some integer $r$ satisfying $0\leq r\lesssim d^2q_k^2$. 

Roughly speaking, the above discussion tells us that for each $k$, we can cover $\Phi(E_k',E_k')$ by at most $Cd^2q_k^2$ intervals of length $cLq_k^{-(d-1)/s}$ for some constants $C,c>0$. This indeed implies that
\begin{align*}
    \dim_H(\Phi(E',E'))\leq \dim_B (\Phi(E',E'))\leq \lim_{k\to \infty}\frac{\log Cd^2q_k^2}{-\log (cLq_k^{-(d-1)/s})}=\frac{2s}{d-1}=\frac{2\dim_H(E')}{d-1}.
\end{align*}
We complete the proof of the theorem.
\end{proof}

\end{document}